\documentclass[12pt]{article}
\usepackage{amsmath,amssymb,amsthm,mathtools}
\usepackage{fourier}
\usepackage{tikz}

\theoremstyle{plain}
\newtheorem{theorem}{Theorem}
\newtheorem{lemma}[theorem]{Lemma}
\newtheorem*{theorem*}{Theorem}

\newcommand{\ka}{\kappa}

\begin{document}

\title{Amicable lattice rhombuses are equable}
\author{Bohdan Biekietov \and Iwan Praton \and Weiran Zeng}
\date{}
\maketitle

\begin{abstract}
A polygon is equable if its area is equal to its perimeter. A pair of polygons is an amicable pair if the area of the first is equal to the perimeter of the second, and vice versa. A polygon is a lattice polygon if its vertices lie on the integer lattice. We show that amicable lattice rhombuses are actually equable. 
\end{abstract}

\begin{section} {Introduction}
A polygon is \emph{equable} if its area is equal to its perimeter. Two polygons form an \emph{amicable pair} if the area of one is equal to the perimeter of the other, and vice versa. Slightly abusing terminology, we say that a polygon is \emph{amicable} if it is one of an amicable pair. Without further conditions, it is easy to find many equable and amicable polygons, so we usually restrict the kind of polygons we consider. In this paper we concentrate on lattice polygons, i.e., polygons whose vertices lie on the integer lattice.

It has long been known that there are only five equable lattice triangles (see the proof in the appendix of \cite{AC1}). It turns out there is only one pair of amicable lattice triangles \cite{PS}. There are two equable lattice rectangles (one of which is a square), and there are five pairs of amicable lattice rectangles \cite{PZ}. In this paper we take a look at rhombuses. There are two equable rhombuses on the integer lattice (\cite{AC1}, Corollary 1(a)); in this paper we investigate \emph{amicable} rhombuses instead.

We first need to be more explicit on what we mean by an amicable pair. An equable polygon paired with itself forms an amicable pair, but it's a trivial pair. 
In this paper, we allow such trivial pairs since they arise from nontrivial solutions to the equations we use below. Our results show that all amicable lattice rhombuses are actually equable. 

We first note some general properties of amicable parallelograms on the integer lattice. Recall that the side lengths of an amicable polygon are integers (Lemma 2 in \cite{PZ}). Suppose we have an amicable pair of parallelograms, one with side lengths $x$ and $y$, and the other with side lengths $a$ and $b$, where $x,y,a,b$ are positive integers. These side lengths completely determine the two parallelograms: if $\theta$ is an interior angle of the first parallelogram, then its area is $xy\sin \theta$, which is equal to $2(a+b)$, so $\sin \theta = 2(a+b)/(xy)$. The angle of the second parallelogram is also determined by $x,y,a,b$. Thus it suffices to describe an amicable pair by providing their side lengths.

\end{section}

\begin{section}{Results}
In their study of equable lattice parallelograms, Aebi and Cairns start by proving a necessary and sufficient condition for a lattice parallelogram to be equable. Somewhat surprisingly, a very minor modification also works for amicable lattice parallelograms. 

\begin{theorem}
Let $x, y, a, b$ be positive integers. Then there exists an amicable pair of lattice parallelograms, one with sides $x$, $y$ and the other with sides $a$, $b$, if and only if both $x^2y^2 -4(a+b)^2$ and $a^2b^2 -4(x+y)^2$ are square integers.
\end{theorem}
The proof is virtually identical with the proof in \cite{AC1}. We provide it here for completeness.
\begin{proof}
Suppose $P$ is a lattice parallelogram with side lengths $x, y$; its amicable partner has side lengths $a, b$. The area of $P$ is $hy$ where $h$ is the height of $P$; see Figure 1 below. By amicability, we have $hy=2(a+b)$, so $h^2y^2=4(a+b)^2$. 

\begin{figure}[h]
\quad \begin{tikzpicture}[scale=0.8]
\draw (0,0)--(10,0)--(13,2)--(3,2)--(0,0);
\draw (3,2)--(3,0);
\draw [dashed] (3,2)--(10,0);
\node at (1.5,1.5) {$x$};
\node at (1.5,-0.2) {$u$};
\node at (6.5,1.3) {$d$};
\node at (7.5,2.3){$y$};
\node at (3.2,1){$h$};
\end{tikzpicture}
\caption{An amicable parallelogram}
\end{figure}
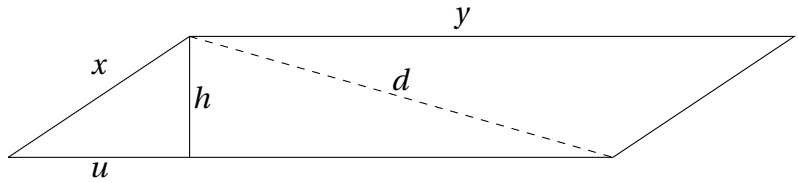

Let $d$ be the length of the short diagonal. Since $P$ is a lattice parallelogram, $d^2$ is an integer. Then
\begin{align*}
d^2 &= h^2 + (y-u)^2  \\
&= h^2 + y^2 + u^2 - 2yu\\
&= h^2 + y^2 + x^2 - h^2 - 2yu\\
&= y^2 + x^2 - 2y \sqrt{x^2 - h^2}\\
&= y^2 + x^2 - \sqrt{4x^2y^2 - 4h^2y^2}\\
&= y^2 + x^2 - \sqrt{4x^2y^2 - 16(a+b)^2};\\
\end{align*}
so $\sqrt{4x^2y^2 - 16(a+b)^2}$ is an integer, and hence $4x^2y^2 - 16(a+b)^2$ is a square. This implies that $x^2y^2 -4(a+b)^2$ is a square. Similarly, $a^2b^2 - 4(x+y)^2$ is a square.

Conversely, suppose that $x, y, a, b$ are positive integers where $x^2y^2 -4(a+b)^2$ and $a^2b^2 -4(x+y)^2$ are squares. Let 
\[d_1=\sqrt{x^2+y^2 + 2\sqrt{x^2y^2 - 4(a+b)^2}}.
\]
Note that $d_1$ is a positive real number since the inner square root is actually a (nonnegative) integer by hypothesis. We construct a triangle $T_1$ whose sides are $x, y$, and $d_1$; such a triangle exists because $x,y<d_1<x+y$. (The inequality $d_1<x+y$ holds because $x^2+y^2 + 2\sqrt{x^2y^2 - 4(a+b)^2} < (x+y)^2$ since $\sqrt{x^2y^2 - 4(a+b)^2} < xy$.) Now let $\theta$ denote the angle between sides $x,y$. Note that $\pi/2\leq \theta < \pi$ since $d_1^2 \geq x^2+y^2$. So
\begin{align*}
d_1^2&=x^2+y^2 - 2xy\cos\theta\\
&= x^2+y^2 + 2xy\sqrt{1-\sin^2\theta}\\
&=x^2+y^2 + 2\sqrt{x^2y^2 - x^2y^2\sin^2\theta}.
\end{align*}
Comparing this with the definition of $d_1$, we see that $x^2y^2\sin^2\theta = 4(a+b)^2$, or $\frac12 xy \sin\theta = a+b$. So the area of $T_1$ is $a+b$. Let $P_1$ be the parallelogram consisting of two copies of $T_1$. Then the area of $P_1$ is $2(a+b)$. 

Similarly, we can construct a triangle $T_2$ with sides $a, b$, and \[ d_2:=\sqrt{a^2+b^2 + 2\sqrt{a^2b^2 - 4(x+y)^2}}.\] Let $P_2$ be the parallelogram consisting of two copies of $T_2$. Then the area of $P_2$ is $2(x+y)$. We see that $P_1$ and $P_2$ form an amicable pair.

It remains to show that $P_1$ and $P_2$ are lattice parallelograms. Equivalently, we need to show that $T_1$ and $T_2$ are lattice triangles. But $x^2, y^2$, and $d_1^2$ are integers, and the area of $T_1$ is an integer. So $T_1$ is a geodetic triangle, in the terminology of Paul Yiu \cite{Y}. Thus $T_1$ can be realized as a lattice triangle. Similarly, $T_2$ can also be realized as a lattice triangle. This completes the proof.
\end{proof}

We note in passing that it is possible for  $x^2y^2 -4(a+b)^2$ to be zero. In this case the parallelogram with side lengths $x$ and $y$ is actually a rectangle.

Theorem 1 recasts our geometry problem as a number theory problem. We can then use the arsenal of number theory to tackle our problem; in this case we will use the theory of Pythagorean triples. 

We now specialize to rhombuses, setting $x=y$ and $a=b$. Theorem 1 says that there exists an amicable pair of rhombuses with side lengths $x$ and $a$ if and only if $x^4 - 16a^2$ and $a^4 - 16x^2$ are both squares. From the theory of Pythagorean triples, there are  positive integers $k$, $m$, $n$ (with $m$ and $n$ coprime), such that 
\[
x^2 = k(m^2+n^2), \quad 4a = 2kmn, \eqno{(1\text{a})}
\]
or
\[
x^2 = k(m^2+n^2), \quad 4a = k(n^2-m^2),\eqno{(1\text{b})}
\]
assuming that $n>m$. In the context of equable lattice parallelograms, Aebi and Cairns showed that (1b) can be reduced to (1a). We follow their lead here. 

Suppose (1b) holds. Then $x^4 - 16 a^2 = 2k^2m^2n^2$, so $x$ is even. We now show that $k$ is even. For a contradiction suppose $k$ is odd. Since $x^2=k(m^2+n^2)$ is even, it must be that $m^2+n^2$ is even, so both $m$ and $n$ must be odd. (They can't both be even since they are coprime.) This means $m^2+n^2 \equiv 2 \pmod{4}$, so $x^2 = k (m^2+n^2) \equiv 2 \pmod{4}$, a contradiction. Thus $k$ is indeed even, say $k=2\kappa$. Write $u=m+n$, $v=n-m$, so $m=(u-v)/2$ and $n=(u+v)/2$. Then
\[
x^2 = k(m^2+n^2) = \kappa(u^2+v^2), \quad 4a = k(n^2-m^2) = 2\kappa uv,
\]
which is of the form (1a). Since $m$ and $n$ are coprime, the gcd of $u$ and $v$ is either 1 or 2. If the gcd is 1, then (1a) holds; if the gcd is 2, then we can write $u=2u'$ and $v=2v'$, where $u'$ and $v'$ are coprime. Then (1a) holds with $u'$ and $v'$ in place of $m$ and $n$ and $4\kappa$ in place of $k$. So from now on we will work with (1a).

Since $4a=2kmn$, we have $2a=kmn$, so one of $k$, $m$, $n$ must be even. Suppose first that $k$ is even, say $k=2\kappa$. Then $a=\kappa mn$. By Theorem 1, $a^4 - 16 x^2$ is a square, say $s^2$. Thus
\[
\kappa^4 m^4n^4 - 32\kappa(m^2+n^2) = s^2. \eqno{(*)}
\]

On the other hand, suppose $k$ is odd. Then one of $m$ or $n$ is even; it does no harm to choose $m$ to be even, say $m=2\mu$. Then $a=k\mu n$, and Theorem 1 again implies that
\[
k^4\mu^4 n^4 - 16k(4\mu^2+n^2)=s^2. \eqno{(**)}
\]
Moving forward, we will work mainly with these equations. We will show that each has only one solution.
\end{section}

\begin{section}%
{A Digression into Elliptic Curves}
It might be of interest to note that equations $(*)$ and $(**)$ are related to elliptic curves. Take $(*)$ for example. It can be rewritten as $\alpha^2\beta^2 - 32(\alpha+\beta)=s^2$, where $\alpha=\ka m^2$ and $\beta=\ka n^2$. Viewing this as a quadratic in $\beta$ with integer roots, we see that the discriminant $32^2 +4\alpha^2(32\alpha+s^2)$ must be the square of an integer, say $z^2$. Multiplying by $2^8$, we get
\[
2^{18} + 2^{15}\alpha^3 + 2^{10}s^2\alpha^2 = 2^8z^2,\quad \text{ or } \quad 2^{18} + (2^{5}\alpha)^3 + s^2 (2^{5}\alpha)^2 = (2^4z)^2,
\]
which is of the form $u^2 = v^3 + cv^2 + d$, a form of an elliptic curve. (In more detail, $u=2^4 x$, $v=2^5\alpha$, $c=s^2$, and $d=2^{18}$.) The linear change of variables $X=9v-3s^2$ and $Y=27u$ produces the standard Weierstra\ss\ form $Y^2=X^3 + pX+q$. We thus have a family of elliptic curves parametrized by $s$. 

We will not pursue this idea further since we aim to solve $(*)$ and $(**)$ by elementary means. We tackle each case separately, but we use very similar methods for both.
\end{section}

\begin{section}{$k$ even}
We first take a look at $(*)$. The equation can be rewritten as
\[
(\ka^2 m^2n^2+s)(\ka^2 m^2n^2 - s) = 32 \ka (m^2+n^2) \eqno{(*')}
\]
Our main idea is that the first factor on the left side of $(*')$ is larger than the right side when $\ka$, $m$, and $n$ are large, so $(*')$ can only be satisfied for reasonably small values of $\ka$, $m$, and $n$, and we can check each of these values. 

We first note that since $\ka^2 m^2n^2+s$ and $\ka^2 m^2n^2-s$ have the same parity. Since the right side of $(*')$ is even, it must be that $\ka^2 m^2n^2+s$ and $\ka^2 m^2n^2-s$ are both even (and positive). We then have the following result.

\begin{lemma}\label{evenlem}
Suppose $\ka^2m^2n^2-s=2\beta$ for some positive integer $\beta$. Then $\ka$ is a divisor of  $\beta^2$ and 
\[
\beta^2 + 64 = (\beta\ka m^2 - 8)(\beta\ka n^2 - 8).
\]
If $\ka^2m^2n^2-s\geq 2\beta$, then $\beta^3 + 64 \geq (\beta\ka m^2 - 8)(\beta\ka n^2 - 8)$.
\end{lemma}
\begin{proof}
Suppose $\ka^2m^2n^2-s=2\beta$. Then $s=\ka^2m^2n^2-2\beta$ and $(*)$ becomes
\[
\ka^4m^4n^4 - (\ka^2m^2n^2-2\beta)^2 = 4\beta\ka^2m^2n^2 -4\beta^2 = 32\ka(m^2+n^2),
\]
so $\beta^2 = \beta \ka^2m^2n^2 -8\ka (m^2+n^2)$, which implies  $\ka \mid \beta^2$. Continuing the calculation, we have
\begin{align*}
\beta^2 &= \beta \ka^2m^2n^2 -8\ka (m^2+n^2)\\
\beta^3 & = \beta^2 \ka^2m^2n^2 -8\beta\ka (m^2+n^2)\\
&= (\beta\ka m^2 - 8)(\beta\ka n^2 - 8) - 64,
\end{align*}
so $\beta^3 + 64 = (\beta\ka m^2 - 8)(\beta\ka n^2 - 8)$, as required.

If we have $\ka^2m^2n^2-s\geq 2\beta$,  then the same calculation shows that $\beta^3 + 64 \geq (\beta\ka m^2 - 8)(\beta\ka n^2 - 8)$.
\end{proof}
The Lemma basically places a bound on how large $\ka, m, n$ could be in order to solve $(*)$. We provide a more concrete bound in the next Lemma.

\begin{lemma}\label{bound1}
If $(\ka, m, n)$ is a solution to $(*)$, then 
\[
(\ka m^2-2)(\ka n^2 -2)\leq 8. \eqno{(2)}
\]
\end{lemma}
\begin{proof}
Our strategy is to show that $\ka^2m^2n^2-s$ cannot be 2, 4, or 6, and then use the second part of Lemma~\ref{evenlem}.

Suppose first that $\ka^2m^2n^2-s=2$. Then by Lemma~\ref{evenlem}, $\ka$ divides 1, so $\ka=1$. The same Lemma then implies $(m^2-8)(n^2-8) = 65 = 5\cdot 13$. Therefore $m^2-8$ is either 1, 5, 13, or 65. The only integer value of $m$ satisfying this condition is $m=3$, but this means $n^2-8=65$, which is impossible. We conclude that $\ka^2m^2n^2 - s >2$.

Now suppose $\ka^2m^2n^2-s=4$. Then $\ka \mid 4$, so $\ka=1, 2$, or $4$.  
Also, $72=(2\ka m^2 - 8)(2\ka n^2 - 8)$, so $(\ka m^2 - 4)(\ka n^2 - 4) =18 = 2\cdot 3^2$. If $\ka=4$, then $16(m^2-1)(n^2-1)=18$, an impossibility. If $\ka=2$, then $4(m^2-2)(n^2-2)=18$, also an impossibility. The only remaining possibility is  $\ka=1$. In this case $m^2-4$ is either 1, 2, 3, 6, 9, or 18, but none of these can be satisfied with an integer value of $m$. We conclude that $\ka^2m^2n^2 - s >4$.

So suppose $\ka^2m^2n^2-s=6$. Lemma~\ref{evenlem} implies that $\ka = 1, 3$, or $9$. We also have $(3\ka m^2-8)(3\ka n^2 -8) = 91 = 7\cdot 13$. If $\ka=9$, then $(3\ka m^2-8)(3\ka n^2 -8) = (27m^2-8)(27n^2-8)\geq 19\cdot 19 > 91$, so $\ka\neq 9$. If $\ka=3$, then $9m^2-8= (3m)^2-8$ is either 1, 7, 13, or 91, but none of these possibilities produces an integer value for $m$. If $\ka=1$, then $3m^2-8$ is either 1, 7, 13, or 91; none of these possibilities produces integer values for $m$. Thus we must have $\ka^2m^2n^2 - s \geq 8$. Lemma~\ref{evenlem} then implies $128 \geq (4\ka m^2 - 8)(4\ka n^2 - 8)$, so $8 \geq (\ka m^2 -2)(\ka n^2 -2)$ as required.
\end{proof}

We are now in a position to solve $(*)$. But we first state a simple little lemma that we will use repeatedly.

\begin{lemma}\label{littlelem}
Suppose $A$, $c$, $d$ are positive integers and $A^2-2cA-d$ is a square. Then $2A\leq (c+1)^2+d$.
\end{lemma}
\begin{proof}
Say $A^2-2cA-d=(A-c)^2-(c+d)$ is a square $s^2$. Then $s^2$ is a square smaller than $(A-c)^2$. The next square below $(A-c)^2$ is $(A-c-1)^2$, so $s^2\leq (A-c-1)^2$. This means $A^2-2cA-d\leq A^2-2(c+1)A + (c+1)^2$, which implies $2A\leq (c+1)^2 +d$.
\end{proof}

\begin{theorem}
The only solution to $(*)$ is $\ka = 4$, $m=1$, $n=1$.
\end{theorem}
\begin{proof}
The left side of (2) in Lemma~\ref{bound1} is either zero, negative, or positive. If it is zero, then one of the factors, say $(\ka m^2-2)$, is zero. This implies that $\ka =2$ and $m=1$. Then $(*)$ becomes $16n^4 - 64(n^2+1)=s^2$, so $16 \mid s^2$ or $4\mid s$. Write $s=4t$, then $n^4 - 4(n^2+1) = t^2$. By Lemma~\ref{littlelem},  $2n^2\leq 3^2+4=13$. This means $n=1$ or $n=2$, but neither of these values satisfies $(*)$. So the left side of $(2)$ is not zero.

Suppose the left side of $(2)$ is negative. Say $\ka m^2-2<0$. This means $\ka=1$ and $m=1$. Then $(*)$ becomes $n^4 - 32n^2 - 32=s^2$. Lemma~\ref{littlelem} then implies that  $2n^2\le  321$. So $n\leq 12$. We can check easily that none of these values of $n$ (with $\ka=1$ and $m=1$) satisfies $(*)$. Thus $\ka m^2-2$ is positive. Similarly, $\ka n^2-2$ is also positive. 

First consider the case $m=n=1$. (Note that $m$ and $n$ being coprime does not rule out $m=n=1$.) Then (2) becomes $(\ka-2)^2\leq 8$, so $\ka$ is either 3 or 4. Now $(\ka,m,n)=(3,1,1)$ does not satisfy $(*)$. If $\ka=4$, we get the solution mentioned in the theorem.

We now consider the case $mn>1$. We can assume that $m>n\geq 1$. If $m\geq 4$, then $(\ka m^2 -2)(\ka n^2-2)\geq 14\cdot 1$, contradicting equation (2). If $m=3$, then $\ka m^2 - 2 = 9\ka -2$, so in order to satisfy equation $(2)$, we must have $\ka=1$. But then $\ka n^2 -2 = n^2 -2$, so $n\geq 2$. Their product violates equation $(2)$. The only remaining possibility is $m=2$ and $n=1$. Then the left side of $(2)$ becomes $(4\ka -2)(\ka-2)$. For both factors to be positive, we must have $\ka\geq 3$, but these values do not satisfy $(2)$. 

Overall, then, the only solution to $(*)$ is $(\ka,m,n)=(4,1,1)$.
\end{proof}

The solution mentioned in the theorem produces $x=4$ and $a=4$, i.e., two equable squares with side lengths 4. 
\end{section}

\begin{section}{$k$ odd}
We now look at $(**)$. It can be rewritten as
\[
(k^2\mu^2n^2 +s)(k^2\mu^2n^2 - s) = 16k(4\mu^2+n^2) \eqno{(**')}
\]
As before, we can conclude that $k^2\mu^2n^2 - s$ is even. We also have the following version of Lemma~\ref{evenlem}.

\begin{lemma}\label{oddlem}
Suppose $k^2\mu^2n^2 - s = 2\beta$ for some positive integer $\beta$. Then
\[
\beta^3 + 64 = (\beta k\mu^2 -4)(\beta k n^2 -16) \eqno{(3)}
\]
and $k$ divides $\beta^2$. 

If  $k^2\mu^2n^2 - s \geq  2\beta$, then
\[
\beta^3 + 64 \geq (\beta k\mu^2 -4)(\beta k n^2 -16)
\]
\end{lemma}
\begin{proof}
The proof is the same as the proof of Lemma (2). If $k^2\mu^2n^2 - s = 2\beta$, then $(**)$ becomes 
\begin{align*}
k^4\mu^4n^4 - s^2 &= k^4\mu^4n^4 - (k^2\mu^2n^2 -2\beta)^2 \\
16k(4\mu^2+n^2)&= 4\beta k^2\mu^2n^2 -4\beta^2,\\
\beta^3 &= \beta^2 k^2\mu^2n^2 - 4k\beta(4\mu^2+n^2)\\
&=(\beta k\mu^2 -4)(\beta k n^2 -16) - 64\\
\beta^3 + 64 &= (\beta k\mu^2 -4)(\beta k n^2 -16),
\end{align*}
as required. The second line shows that $k$ is a divisor of $\beta^2$. 

The same calculation proves the second part of the Lemma.
\end{proof}

We also have the following version of Lemma~\ref{bound1}.

\begin{lemma}\label{bound2}
If $(\ka, m, n)$ is a solution to $(**)$, then 
\[
(5k\mu^2-4)(5k n^2 -16)\leq 189. \eqno{(4)}
\]
\end{lemma}
\begin{proof}
As before, we show that $k^2\mu^2n^2 - s \geq 10$ and then use Lemma~\ref{oddlem}. We use a somewhat different method for the proof, utilizing modular reduction instead of factorization. 

Suppose first that $k^2\mu^2n^2 - s = 2$. Then Lemma~\ref{oddlem} says that $k=1$ and that $(\mu^2-4)(n^2-16)=65$. We reduce this equation modulo 3, getting $(\mu^2-1)(n^2-1)\equiv -1 \pmod{3}$. A square mod 3 is either 0 or 1, so $\mu^2-1$ and $n^2-1$ are each either $-1$ or $0$, hence their product cannot be $-1$ mod 3.

Suppose now that $k^2\mu^2n^2 - s = 4$. Then Lemma~\ref{oddlem} says that $k$ is either 1, 2, or 4; since $k$ is odd, it must be that $k=1$. The lemma also says that $(2\mu^2-4)(2n^2-16)= 72$, or $(\mu^2 - 2)(n^2 - 8) = 18$.  We again reduce modulo 3; then $\mu^2-2 \equiv 1$ or $2 \pmod{3}$, and similarly with $n^2-8$. Their product is never $0 \pmod{3}$. 
Thus there is no solution with both $\mu$ and $n$ integers. 

Next, suppose $k^2\mu^2n^2 - s = 6$. Lemma~\ref{oddlem} implies that $k=1, 3$, or $9$, and that $(3k\mu^2 - 4)(3kn^2 - 16) = 91 = 7\cdot 13$. If $k=9$, the left side is at least $23\cdot 11 > 91$, contradicting equation (3). If $k=3$, we have $(9\mu^2 - 4)(9n^2 -16) = 91$. Clearly $n>1$, but then $(9\mu^2 - 4)(9n^2 -16)\geq 5\cdot 20 > 91$, again contradicting equation (3). If $k=1$, then $(3\mu^2 - 4)(3n^2 -16)=91$. Reducing modulo 4, we get $(\mu n)^2 \equiv -1 \pmod{4}$, a contradiction. 

Finally, suppose that $k^2\mu^2n^2 - s = 8$. In this case Lemma~\ref{oddlem} implies that $k=1$ or $k=2^u$, $1\leq u \leq 4$. Since $k$ is odd, we must have $k=1$. The Lemma also implies that $(4\mu^2-4)(4n^2-16) = 128$, or $(\mu^2-1)(n^2-4)=8$. We reduce modulo 3. Now $n^2-4\equiv n^2-1\equiv 0 \text{ or } 2 \pmod{3}$, and similarly for $\mu^2-1$. Their product is either $0$ or $1 \pmod{3}$, but $8\equiv 2 \pmod{3}$. So there is no integer solution in this case, hence $k^2\mu^2n^2 - s \ge 10$. The second part of Lemma~\ref{oddlem} then provides the result.
\end{proof}

\begin{theorem}
The only solution to $(**)$ is $k=5, \mu=1, n=1$.
\end{theorem}
\begin{proof}
First, note that $n$ has to be odd, since $n$ and $m=2\mu$ are coprime. Our main tool is equation (4) in Lemma~\ref{bound2}. If $k\geq 6$, then $(5k\mu^2-4)(5kn^2-16)\geq 26\cdot 14 > 189$, contradicting equation (4). So $k$ must be either $1$, $3$, or $5$. 

If $k=1$, then equation (4) becomes $(5\mu^2-4)(5n^2-16)\leq 189$. If $n=1$, this inequality is automatically satisfied for all values of $\mu$.  But in this case $(**)$ becomes $\mu^4 - 64\mu^2-16=s^2$, and Lemma~\ref{littlelem} implies that $2\mu^2 \le 1105$, which means $\mu\leq 23$.  We check all these values; none satisfies $(**)$. So there is no solution with $k=1$ and $n=1$.

We now try $k=1$, $\mu=1$. For a brief change of pace we skip the use of Lemma~\ref{littlelem}. In this case $(**)$ becomes $n^4 - 16n^2 -64=s^2$, or $(n^2-8)^2 - s^2 = 128$. Thus $(n^2-8 + s)(n^2-8-s)=128$. Since 128 can be factored as $128\cdot 1$, $64\cdot 2$, $32\cdot 4$, and $16\cdot 8$, the possible values of $n^2-8$ are $129/2$, $66/2$, $36/2$, and $24/2$. None of these produce integer values of $n$. Thus there is no solution with $k=1$ and $\mu=1$.

Therefore if $k=1$, we must have $n\geq 3$ and $\mu\geq 2$. This means the left side of (4) is $(5\mu^2 - 4)(5n^2-16) \geq 16\cdot 29 > 189$, so these values of $n$ and $\mu$ do not satisfy (4). 

Next we look at $k=3$. Here equation (4) becomes $(15\mu^2 -4)(15n^2-16)\le 189$. If $n\geq 3$, then the left side is at least $11\cdot 119 >189$, so  there is no solution with $k=3$ and $n\geq 3$. It remains to tackle the troublesome case $n=1$.  Here $(**)$ becomes $81\mu^4 - 192 \mu^2 - 48 = s^2$. We can't really use Lemma~\ref{littlelem}, since $192/9$ is not an integer. A small modification works. We have $s^2=(9\mu^2-10)^2 - 12\mu^2 - 148$, so $s^2$ is at most $(9\mu^2-11)^2= 81\mu^4 - 198 \mu^2 + 121$. Consequently $81\mu^4 - 192 \mu^2 - 48 \leq 81\mu^4 - 198 \mu^2 + 121$, or $6\mu^2 \leq 169$. This means $\mu\leq 5$. We check easily that $k=3$, $n=1$, and $\mu\leq 5$ do not satisfy $(**)$. Thus we rule out $k=3$. 

When $k=5$,  equation (4) becomes $189 \geq (25\mu^2 - 4)(25n^2-16)$. If either $\mu>1$ or $n>1$, then the right side is larger than $(25-4)(25-16)=189$, a contradiction. So $\mu=1$ and $n=1$. We check that $(k,\mu,n)=(5,1,1)$ does satisfy $(**)$.  This concludes the proof.
\end{proof}

This solution produces $x=5$ and $a=5$, two equable rhombuses with side length 5 and area 20. In conclusion, then, all amicable lattice rhombuses are actually equable.

\end{section}

\bigskip
\noindent \textsc{Franklin \& Marshall College}\\
ipraton@fandm.edu

\end{document}